\documentclass[12pt]{amsart}
\usepackage{amsmath, amscd}
\usepackage{amsfonts}
\usepackage{version}
\usepackage{amssymb}
\newcommand{\de}{\partial}
\newcommand{\R}{\mathbb R}

\newcommand{\al}{\alpha}

\newcommand{\C}{\mathbb C}

\newcommand{\B}{\mathbb B}

\newcommand{\D}{\mathbb D}

\newcommand{\N}{\mathbb N}

\def\v{\varphi}

\def\Re{{\sf Re}\,}

\newtheorem{theorem}{Theorem}[section]
\newtheorem{lemma}[theorem]{Lemma}
\newtheorem{proposition}[theorem]{Proposition}
\newtheorem{corollary}[theorem]{Corollary}
\newtheorem{question}[theorem]{Question}

\theoremstyle{definition}
\newtheorem{definition}[theorem]{Definition}
\newtheorem{example}[theorem]{Example}

\theoremstyle{remark}
\newtheorem{remark}[theorem]{Remark}
\numberwithin{equation}{section}

\numberwithin{equation}{section}

\begin{document}
\title[Variation, extreme and support points]{Variation of Loewner chains,
extreme and support points in the class $S^0$  in higher dimensions}
\author[F. Bracci]{Filippo Bracci$^\dag$}\thanks{$^\dag$ Supported by
the ERC grant ``HEVO - Holomorphic Evolution Equations'' n. 277691}
\author[I. Graham]{Ian Graham$^\ddag$}\thanks{$^\ddag$ Supported by the
Natural Sciences and Engineering Research Council of Canada under
grant A9221}
\author[H. Hamada]{Hidetaka Hamada$^{\star}$}\thanks{$^\star$ Supported by
JSPS KAKENHI Grant Number 25400151}
\author[G. Kohr]{Gabriela Kohr$^{\star\star}$}\thanks{$^{\star\star}$
Supported by a grant of the
Romanian National Authority for Scientific Research, CNCS-UEFISCDI,
project number PN-II-ID-PCE-2011-3-0899}
\address{F. Bracci: Dipartimento di Matematica\\
Universit\`{a} di Roma \textquotedblleft Tor Vergata\textquotedblright\ \\
Via Della Ricerca Scientifica 1, 00133 \\
Roma, Italy} \email{fbracci@mat.uniroma2.it}
\address{I. Graham: Department of Mathematics\\
University of Toronto\\
40 St. George Street\\
Toronto, Canada M5S 2E4} \email{graham@math.toronto.edu}
\address{H. Hamada: Faculty of Engineering, Kyushu Sangyo University, 3-1 Matsukadai 2-Chome,
Higashi-ku Fukuoka 813-8503, Japan}
\email{h.hamada@ip.kyusan-u.ac.jp}
\address{G. Kohr: Faculty of Mathematics and Computer Science,
Babe\c{s}-Bolyai University, 1 M. Kog\u{a}l\-niceanu Str., 400084
Cluj-Napoca, Romania} \email{gkohr@math.ubbcluj.ro}


\date{}


\begin{abstract}
We introduce a family of natural normalized Loewner chains in the unit ball,
which we call ``ger\"aumig''---spacious---which allow to construct,
by means of suitable variations, other normalized Loewner chains which coincide
with the given ones from a certain time on. We apply our construction to the study
of support points, extreme points and time-$\log M$-reachable functions in the class $S^0$ of
mappings admitting parametric representation.
\end{abstract}

\maketitle

\section{Introduction}

Let $\B^n:=\{z\in \C^n: \|z\|^2<1\}$ denote the Euclidean unit ball of $\C^n$. Let
\[
S:=\{f: \B^n \to \C^n: f(0)=0, df_0={\sf id}, f \hbox{ univalent}\}
\]
be the class of normalized univalent mappings in $\B^n$.
For $n=1$ the class $S$ is compact, and a great variety of extremal problems have been
studied (see {\sl e.g.} \cite{Du}, \cite{HM}, \cite{Po}, \cite{Ro}, \cite{Scha}).
Also, in the case of one complex variable every $f\in S$ can be embedded into a
normalized Loewner chain (see \cite{Po}).
Much is known about the structure of extreme points and support points of linear problems,
in particular they are single-slit mappings (see {\sl e.g.} \cite[pp. 286-288 and pp. 306-307]{Du}).

In higher dimensions, the class $S$ is not compact,
there are no single-slit mappings,
and it is not known whether every element in $S$
can be embedded into a normalized Loewner chain.
Partial results concerning the latter question
can be found in \cite{ABW2}, \cite{GHK1}, \cite{GKK}.

For $n>1$ the compact subclass $S^0$ of $S$ of mappings admitting parametric
representation was introduced in \cite{GHK1}; it was first considered by Poreda
(see \cite{Por1}, \cite{Por2}) on the polydisc. As in the case of one complex variable, the
class $S^0$ does not have a linear structure, so it natural to consider both linear and
nonlinear extremal problems in the class $S^0$. In this paper we will focus on linear
problems where some recent progress has been made in identifying mappings which are or
which are not support points or extreme points
(see \cite{B}, \cite{GHKK1}, \cite{GHKK2}, \cite{Ro2}, \cite{Schl}).
However, we point out that our method works for non-linear problems as well,
and we give an account of it in Proposition \ref{nonlinear}.

In his 1998 thesis, Roth \cite{Ro} developed a control-theoretic variational method which
gives a version of Pontryagin maximum principle for families of holomorphic functions on the
unit disc. Recently, he  obtained similar results in higher dimensions for mappings on the
unit ball $\B^n$ in $\C^n$ (see \cite{Ro2}).

A variational method for linear invariant families
in ${\sf Hol}(\B^n,\C^n)$ may be found in \cite{PS}.

One of the main difficulties when dealing with univalent mappings in higher dimensions
is that the lack of an uniformization theorem does not allow to construct easily
variations of a given normalized Loewner chain.

The aim of the present paper is to define a natural class of normalized Loewner chains,
which we call {\sl ger\"aumig}, which allow to construct other normalized Loewner chains
having the property that from a certain time on, they coincide with the
initial ger\"aumig Loewner chain.
This variational method seems to be completely new and seems to adapt well to the
case of bounded univalent mappings of the ball having some regular extension up to the boundary.

We refer the reader to Section \ref{ger} for the definition of ``ger\"aumig'' Loewner
chains and Theorem \ref{variation} for the result on variation of ``ger\"aumig'' Loewner chains.
For the time being, we content ourselves to give the following definition.
A normalized Loewner chain $(f_t)_{t\geq 0}$ on $\B^n$ is called
{\sl exponentially squeezing} in $[T_1,T_2)$ for some $0\leq T_1<T_2\leq \infty$
provided there exists  $a\in (0,1]$ such that for all $T_1\leq s<t<T_2$
it follows $\|f_t^{-1}(f_s(z))\|\leq e^{a(s-t)}\|z\|$, for all $z\in \B^n$.
In particular we have the following result (whose proof is in Section \ref{support}):

\begin{proposition}
\label{p.support-extreme}
Let  $(f_t)_{t\geq 0}$ be a  normal Loewner chain which is exponentially
squeezing in $[T_1,T_2)$ for some $0\leq T_1<T_2\leq +\infty$.
Then $f_0$ is not a support point of $S^0$. Also, $f_0$ is not an extreme point of $S^0$.
\end{proposition}

A similar result holds in case of Fr\'echet differentiable functionals, see Proposition \ref{nonlinear}.

In fact, an exponentially squeezing normal Loewner chain can be
suitably re-para\-mete\-ri\-zed in time in order to construct a ger\"aumig Loewner
chain (see Theorem \ref{squeez-geraumig}). It is interesting to note that
all bounded normalized functions in the unit disc can be embedded into an
exponentially squeezing chain (and in fact ger\"aumig Loewner chain)
from a certain time on---which, geometrically, amounts to evolve the image
of the mapping into a disc
in finite time and consider then the natural radial dilatation of such a disc.
While, in higher dimensions, this is no longer the case (see Example \ref{ese-no}).

Proposition \ref{p.support-extreme} allows to prove directly the following result
(precise definitions and the proof are contained in Section \ref{support}):

\begin{theorem}
Let $f\in S^0$. Assume that
there exist $g\in S^0$ and $r\in (0,1)$ such that $f(z)=\frac{1}{r}g(rz)$, for all $z\in\B^n$.
Then $f$ is not a support point of the class $S^0$.
Also, $f$ is not an extreme point of $S^0$.
\end{theorem}

Finally, in Section \ref{Sreach} we apply our results to study time-$\log M$-reachable
mappings and their geometric counterparts, the mappings that can be evolved in finite time to a ball.
As a result, and in neat contrast to the one-dimensional case,
we find an example of a family of mappings $\Phi^N\in S^0$, $N>2$, which are bounded by a constant $M>1$,
are not  support points, nor  extreme points of $S^0$ but cannot be reached in time
$\log R$ for all $2<R<N$. Those mappings $\Phi^N$ are however reachable in time
$\log N$ and are in fact support points of the set of time $\log N$-reachable mappings
(see Theorem \ref{PhiN}).

\section{Ger\"{a}umig Loewner chains}\label{ger}

\subsection{Subordination chains, Loewner chains and parametric representation}
In what follows we denote by $\R^+$ the semigroup of nonnegative real numbers,
and by $\N$ the semigroup of nonnegative integer numbers.

Let
\[\mathcal M:=\{h\in {\sf Hol}(\B^n, \C^n): h(0)=0, dh_0={\sf id},
\Re \langle h(z), z\rangle >0,  \forall z\in \B^n\setminus\{0\}\},\]
where $\langle \cdot, \cdot \rangle$ denotes the Euclidean inner product in $\C^n$.
Applications of this family in the study of biholomorphic mappings
on $\B^n$ and the Loewner theory in higher dimensions
may be found in \cite{Ar}, \cite{ABW2}, \cite{B2}, \cite{DGHK}, \cite{ERS},
\cite{GHK1}, \cite[Chapter 8]{GK03},
\cite{Pf}, \cite{Su77}, \cite{Vo}. A new geometric approach of Loewner theory
on the unit disc and complete hyperbolic manifolds may be found in \cite{B2} and \cite{B1}.

\begin{remark}\label{simply}
If $h\in {\sf Hol}(\B^n, \C^n)$ is such that $h(0)=0$, $dh_0={\sf id}$,
$\Re \langle h(z), z\rangle \geq 0$ for all $z\in \B^n$, then, in fact, $h\in \mathcal M$, by
the minimum principle for harmonic functions.
\end{remark}

\begin{definition}\label{Herglotz}
 A \textit{Herglotz vector field associated with the class
 $\mathcal M$} on $\B^n$ is a mapping $G:\B^n\times \R^+\to
\C^n$ with the following properties:
\begin{itemize}
\item[(i)] The mapping $G(z,\cdot)$ is measurable on $\R^+$ for all
$z\in \B^n$.
\item[(ii)] $-G(\cdot, t)\in \mathcal M$  for a.e. $t\in [0,+\infty)$.
\end{itemize}
\end{definition}

\begin{definition}\label{regular}
For a given Herglotz vector field $G(z,t)$ associated with the class $\mathcal M$ on $\B^n$,
a {\sl normalized solution} to the  {\sl Loewner-Kufarev PDE}
associated with $G(z,t)$ consists of a family $(f_t)_{t\geq 0}$ of
holomorphic mappings from $\B^n$ to $\C^n$ such that
$f_t(0)=0$ and $d(f_t)_0=e^t{\sf id}$ for all $t\geq 0$, the mapping
$t\mapsto f_t$ is continuous with respect to the topology in ${\sf Hol}(\B^n,\C^n)$
induced by the uniform convergence on compacta in $\B^n$, and
the following equation is satisfied for a.e. $t\geq 0$ and for all  $z\in \B^n$
\begin{equation}\label{LoewnerPDE}
\frac{\de f_t}{\de t}(z)=-d(f_t)_z \cdot G(z,t).
\end{equation}
\end{definition}

\begin{definition}
A {\sl normalized subordination chain} $(f_t)_{t\geq 0}$ is a family of holomorphic mappings
$f_t:\mathbb B^n\rightarrow\mathbb C^n$, such that $f_t(0)=0$, $d(f_t)_0=e^t {\sf id}$ for all $t\geq 0$,
and for every $0\leq s\leq t$ there exists $\v_{s,t}:\B^n \to \B^n$ holomorphic such that
$\v_{s,t}(0)=0$ and $f_s=f_t\circ \v_{s,t}$.
A normalized subordination chain $(f_t)_{t\geq 0}$ is called a {\sl normalized Loewner chain}
if for all $t\geq 0$ the mapping $f_t$ is univalent.
\end{definition}

\begin{definition}
A normalized Loewner chain $(f_t)_{t\geq 0}$ on $\B^n$ is called a
{\sl normal Loewner chain} if the family $\{e^{-t}f_t(\cdot)\}_{t\geq 0}$ is normal.
\end{definition}

Putting together \cite[Chapter 8]{GK03}, \cite{GKP}, \cite[Proposition 2.6]{ABW}
(see also \cite{ABHK}, \cite{DGHK}),  we have the following result:

\begin{theorem}\label{all}
\begin{enumerate}
  \item If $(f_t)_{t\geq 0}$ is any normalized Loewner chain on $\B^n$,
  then it is a normalized solution to a Loewner-Kufarev PDE \eqref{LoewnerPDE}
  for some Herglotz vector field $G(z,t)$ associated with the class $\mathcal M$ in $\B^n$.
  \item Let $G(z,t)$ be a Herglotz vector field associated with
  the class $\mathcal M$ on $\B^n$. Then there exists a unique
  normal Loewner chain $(g_t)_{t\geq 0}$---called the {\sl canonical solution}---which is a
  normalized solution to \eqref{LoewnerPDE}. Moreover, $\bigcup_{t\geq 0}g_t(\B^n)=\C^n$.
  \item If $(f_t)_{t\geq 0}$ is a normalized solution to \eqref{LoewnerPDE},
  then $(f_t)_{t\geq 0}$ is a normalized subordination chain on $\B^n$.
  Moreover, there exists a holomorphic mapping $\Phi: \C^n \to \bigcup_{t\geq 0} f_t(\B^n)$,
  with $\Phi(0)=0$ and $d\Phi_0={\sf id}$ such that $f_t=\Phi \circ g_t$, where $(g_t)_{t\geq 0}$
  is the canonical solution to \eqref{LoewnerPDE}.
  In particular, $(f_t)_{t\geq 0}$ is a normalized Loewner chain if and only if $\Phi$ is univalent.
\end{enumerate}
\end{theorem}

We close this section with the notion of parametric representation on $\B^n$
(see \cite{GHK1}; cf. \cite{Por1}, \cite{Por2}, on the unit polydisc in $\C^n$).

\begin{definition}
Let $f\in S$. We say that $f$ admits {\sl parametric representation} if
$$f(z)=\lim_{t\to\infty}e^t\varphi(z,t)$$
locally uniformly on $\B^n$, where
$\varphi(z,0)=z$ and
\begin{equation}\label{LLODE}
\frac{\partial \varphi}{\partial t}(z,t)=G(\varphi(z,t),t),\quad \mbox{ a.e. }\, t\geq 0,\quad
\forall\, z\in\B^n,
\end{equation}
for some Herglotz vector field $G$ associated with the class ${\mathcal M}$ on $\B^n$.
\end{definition}

We denote by $S^0$ the subset of $S$ consisting of mappings which admit parametric representation.

\begin{remark}
(i) It was shown in \cite{GKK} (see also \cite{GHK1}; cf. \cite{Por1}, \cite{Por2}) that
$f\in S$ has parametric representation if and only if
there exists a normal Loewner chain $(f_t)_{t\geq 0}$ on $\B^n$
such that $f_0=f$.

(ii) It is known that $S^0$ is compact in the topology of uniform
convergence on compacta, and that $S^0\neq S$ for $n\geq 2$ (see \cite{GKK}; see also
\cite{GHK1} and \cite{GK03}).
\end{remark}

\subsection{Exponentially squeezing and ger\"{a}umig Loewner chains}
We start with a proposition.

\begin{proposition}\label{bounded}
Let  $(f_t)_{t\geq 0}$ be a normalized Loewner chain on $\B^n$.
Let $0\leq T_1<T_2\leq \infty$ and $a\in (0,1]$. The following conditions are equivalent:
 \begin{enumerate}
   \item  For a.e. $t\in [T_1,T_2)$ and for all $z\in \B^n\setminus\{0\}$ it holds
\begin{equation}\label{squeezing}
      \Re \left\langle [d(f_t)_z]^{-1}\frac{\de f_t}{\de t}(z), \frac{z}{\|z\|^2}\right\rangle \geq a.
\end{equation}
   \item For all $T_1\leq s<t<T_2$ it holds
\begin{equation}\label{equiv}
\|f_t^{-1}(f_s(z))\|\leq e^{a(s-t)}\|z\|,  \quad \hbox{for all $z\in \B^n$}.
\end{equation}
 \end{enumerate}
Moreover, if one of the previous conditions---and hence both---is satisfied
then $f_t$ is bounded for all $t\in [0,T_2)$ and $\overline{f_s(\B^n)}\subset f_t(\B^n)$
for all $T_1\leq s< t<T_2$.
\end{proposition}

\begin{proof}
Let $(\v_{s,t}:=f^{-1}_t\circ f_s)_{0\leq s\leq t}$ be the {\sl evolution family}
associated with $(f_t)_{t\geq 0}$
(see, {\sl e.g.}, \cite{B2}, \cite{GK03})
and let $G(z,t)=-[d(f_t)_z]^{-1}\frac{\de f_t}{\de t}(z)$
be the associated Herglotz vector field. Then $(\v_{s,t})$ is the unique solution to the Loewner ODE
\begin{equation}\label{ODE}
\frac{\de \v_{s,t}}{\de t}(z)=G(\v_{s,t}(z),t), \quad \hbox{a.e. } t\geq s, \quad \forall z\in \B^n,
\end{equation}
such that $\varphi_{s,s}(z)=z$.
Assume first that \eqref{equiv} holds. Let $T_1\leq s<t<T_2$.
Fix $\eta>0$  and let $w=\varphi_{s,t}(z)$.
We have
\begin{equation}
\label{difference}
\frac{\varphi_{s,t+\eta}(z)-\varphi_{s,t}(z)}{\eta}
=
\frac{\varphi_{t,t+\eta}(w)-w}{\eta},
\quad z\in \B^n,\, t\in (s,T_2).
\end{equation}
Since the limit on the left-hand side of
\eqref{difference}
exists for $\eta\to 0^+$ and is equal to
$\frac{\partial \varphi_{s,t}}{\partial t}(z)$ for a.e. $t\geq s$,
the limit of the right-hand side of (\ref{difference}) also exists for $\eta\to 0^+$.
Using (\ref{ODE}) and (\ref{difference}), we conclude that
\[
\lim_{\eta\to 0^+}\frac{\varphi_{t,t+\eta}(w)-w}{\eta}=G(\varphi_{s,t}(z),t),
\quad \forall z\in \B^n,
\quad \hbox{a.e. }  t\in (s,T_2).
\]
On the other hand,
since $\| \varphi_{t,t+\eta}(w)\|\leq e^{-a\eta}\| w\|$,
in view of the above relation,
\begin{equation}\label{sub}
\Re \langle G(\v_{s,t}(z),t), \v_{s,t}(z)\rangle\leq  -a \|\v_{s,t}(z)\|^2,
\quad \forall z\in \B^n,
\quad \hbox{a.e. }  t\in (s,T_2).
\end{equation}

Let $\mathbb{Q}_+$ be the set of nonnegative rational numbers and
let $\lambda$ be the usual Lebesgue measure in $\mathbb{R}$. Then
for each $s_k\in \mathbb{Q}_+\cap (T_1,T_2)$, there exists $N_k\subset
(s_k,T_2)$ such that $\lambda(N_k)=0$ and
\begin{equation}
\label{sub-k} \Re\langle G(\v_{s_k,t}(z),t), \v_{s_k,t}(z)\rangle\leq  -a \|\v_{s_k,t}(z)\|^2,
\quad \forall t\in (s_k,T_2)\setminus N_k,
\end{equation}
by (\ref{sub}). Let $N=\displaystyle\bigcup_{k\in\mathbb{N}}N_k$.
Then $\lambda(N)=0$ and if $t\in (T_1,T_2)\setminus N$ is fixed, we
deduce in view of (\ref{sub-k}) that
$$\Re\langle G(\v_{s_k,t}(z),t), \v_{s_k,t}(z)\rangle\leq  -a \|\v_{s_k,t}(z)\|^2,\, z\in \B^n,
s_k\in \mathbb{Q}_+\cap (T_1,T_2),\, s_k< t,\, k\in \mathbb{N}.$$ Further,
letting $\{s_{\nu(k)}\}_{k\in\mathbb{N}}\subset \mathbb{Q}_+\cap (T_1,T_2)$,
be an increasing sequence which converges to $t$ in the above relation and using the fact that
$s\mapsto \v_{s,t}(z)$ is continuous on $[0,t]$, we conclude that
$\Re\langle G(z,t),z\rangle\leq -a\| z\|^2$, $z\in \B^n$ for all $t\in (T_1,T_2)\setminus N$.
Thus, \eqref{squeezing} holds.

Conversely, assume that $(f_t)_{t\geq 0}$ satisfies \eqref{squeezing}.
Fix $z\in \B^n\setminus\{0\}$. Let $t_0\in (T_1,T_2)$. Let $s\in [T_1,t_0)$.
Note that for a.e. $t\geq s$ we have
\begin{equation}\label{deri}
\frac{\de \|\v_{s,t}(z)\|^2}{\de t}=
2\Re \left\langle \frac{\de \v_{s,t}}{\de t}(z), \v_{s,t}(z)\right\rangle=
2\Re \langle G(\v_{s,t}(z),t), \v_{s,t}(z)\rangle.
\end{equation}
Then for a.e. $t\in [s,t_0]$ by \eqref{squeezing} and \eqref{deri},
we have $\frac{\de \|\v_{s,t}(z)\|^2}{\de t}/\|\v_{s,t}(z)\|^2 \leq -2a $.
Integrating in $t$ between $s$ and $t_0$, we obtain  $\|\v_{s,t_0}(z)\|^2\leq e^{2a(s-t_0)}\|z\|^2$.
Therefore, $\|\v_{s,t_0}(z)\|\leq e^{a(s-t_0)}\|z\|$ for all $s\in [T_1,t_0]$.
This implies \eqref{equiv}.

Finally note that for $T_1\leq s<t< T_2$,
\[
\overline{f_s(\B^n)}=\overline{f_{t}(\v_{s,t}(\B^n))}\subset \overline{f_{t}(e^{a(s-t)}\B^n})=
f_{t}(\overline{e^{a(s-t)}\B^n}).
\]
Hence $\overline{f_s(\B^n)}\subset f_t(\B^n)$ for all $T_1\leq s< t<T_2$.
Moreover, since $f_t(\B^n)\subseteq f_{T_1}(\B^n)$ for all $t\in [0,T_1]$,
it follows also that $f_t(\B^n)$ is bounded for all $t\in [0,T_2)$.
\end{proof}

\begin{remark}
Let $(f_t)_{t\geq 0}$ be a normalized Loewner chain on $\B^n$.
Let $a>0$ and $0\leq T_1<T_2\leq+\infty$. Suppose that for all $T_1\leq s<t<T_2$ equation \eqref{equiv}
holds. Taking into account that $f_t(0)=0$ and $d(f_t)_0=e^t{\sf id}$, multiplying by $\|z\|$
both sides of \eqref{equiv} and taking the limit for $z\to 0$ we immediately obtain $a\leq 1$.
\end{remark}

\begin{definition}
Let $(f_t)_{t\geq 0}$ be a normalized Loewner chain in $\B^n$.
We say that $(f_t)_{t\geq 0}$ is  {\sl exponentially squeezing  in $[T_1,T_2)$,
for $0\leq T_1< T_2\leq +\infty$ (with squeezing ratio  $a\in (0,1]$})
if condition \eqref{squeezing}---or equivalently \eqref{equiv}---holds.
\end{definition}

\begin{example}
Given $0\leq T_1<T_2\leq +\infty$, examples of normal Loewner chains which are exponentially
squeezing in $[T_1,T_2)$ can be constructed as follows.
Let $G_1(z)=-z$ and let $G_2(z)=-(z_1p_1(z_1),\ldots, z_np_n(z_n))$
where $p_j:\D \to \C$  are holomorphic functions such that $\Re p_j>0$ for $j=1,\ldots, n$.
Let $\theta:\R^+\to [0,1]$ be any measurable function such that $\theta(t)\equiv 1$
for all $t\in [T_1,T_2)$. For instance one can take $\theta(t)=0$ for
$t\in \R^+\setminus [T_1,T_2)$,
or, if $T_1>0$, one can take $\theta$ to be a $C^\infty$ function with compact support in
$(T_1-\epsilon, T_2+\epsilon)$ for $\epsilon>0$ very small.
Then define $G(z,t):=\theta(t) G_1(z)+(1-\theta(t))G_2(z)$. It is easy to see that
$G$ is a Herglotz vector field associated with the class $\mathcal M$. Then by construction,
the canonical solution $(g_t)$ to \eqref{LoewnerPDE} is a normal Loewner chain which is exponentially
squeezing in $[T_1,T_2)$. Notice also that $(g_t)$ might not be exponentially
squeezing in $\R^+\setminus[T_1,T_2)$, as is the case if $p_1$ is the Cayley
transform and $\theta(t)=0$ for $t\in \R^+\setminus [T_1,T_2)$.
\end{example}

In order to properly introduce ger\"aumig Loewner chains, we need some preliminaries of linear algebra.
\begin{definition}
Let $A$ be an $n\times n$ matrix. We let
\[
\mu(A):=\min_{\|v\|=1}\|A(v)\|.
\]
\end{definition}

\begin{lemma}\label{linear}
Let $A$ be an invertible $n\times n$ matrix. Then
\[
\mu(A)=\frac{1}{\|A^{-1}\|}.
\]
\end{lemma}

The proof of the above result is elementary and we omit it.

\begin{definition}\label{spacious}
Let $(f_t)_{t\geq 0}$ be a normalized Loewner chain on $\B^n$.  We say that $(f_t)_{t\geq 0}$ is
{\sl ger\"{a}umig\footnote{``ger\"{a}umig'' is a German word which means ``spacious''}  in $[T_1,T_2)$},
for some $0\leq T_1<T_2\leq +\infty$, if there exist  $a,b>0$ such that
\begin{enumerate}
  \item for all $t\in [T_1,T_2)$ and for all $z\in \B^n$ it holds $\mu(d(f_t)_z)\geq a$,
  \item for a.e. $t\in [T_1,T_2)$ and for all $z\in \B^n$ it holds
 $\left\|\frac{\de f_t}{\de t}(z)\right\|\leq b$,
  \item $(f_t)_{t\geq 0}$ is exponentially squeezing in $[T_1,T_2)$.
\end{enumerate}
We say that $(f_t)_{t\geq 0}$ is {\sl ger\"{a}umig} if it is ger\"{a}umig in $[0,+\infty)$.
\end{definition}

\begin{remark}
Let $a, b$ be as in Definition \ref{spacious} and assume that the squeezing ratio
of $(f_t)_{t\geq 0}$ is $c\in (0,1]$. Then
\[
c\leq \Re\left\langle [d(f_t)_z]^{-1}\frac{\de f_t}{\de t}(z),\frac{z}{\|z\|^2}\right\rangle\leq
\|[d(f_t)_z]^{-1}\|\left\|\frac{\de f_t}{\de t}(z)\right\|\frac{1}{\|z\|}\leq \frac{b}{a}.
\]
\end{remark}

\begin{remark}\label{equiv-cond}
Let $(f_t)_{t\geq 0}$ be a normalized Loewner chain which satisfies the Loewner-Kufarev PDE
\eqref{LoewnerPDE}. If $(f_t)_{t\geq 0}$ satisfies conditions (1) and (3) of Definition \ref{spacious}
and moreover there exists $c>0$ such that
\begin{itemize}
  \item[(2')] for a.e. $t\in [T_1,T_2)$ and for all $z\in \B^n$ it holds $\|d(f_t)_z\|\leq c$
  and $\|G(z,t)\|\leq c$,
\end{itemize}
then by the Loewner-Kufarev PDE it is easy to see that $(f_t)_{t\geq 0}$ satisfies also (2) of
Definition \ref{spacious} and it is therefore ger\"aumig in $[T_1,T_2)$.
\end{remark}

\begin{remark}\label{back}
Suppose $(f_t)_{t\geq 0}$ is a normalized Loewner chain, {\it respectively}
a normal Loewner chain, on $\B^n$, which is  ger\"{a}umig in $[T_1,T_2)$ for some $0<T_1<T_2\leq +\infty$.
Let $\tilde{f}_t(z):=e^{-T_1}f_{T_1+t}(z)$ for $z\in \B^n$.
Then it is easy to check that $(\tilde{f}_t)_{t\geq 0}$ is a normalized Loewner chain, {\it respectively}
a normal Loewner chain, on $\B^n$, which is  ger\"{a}umig in $[0,T_2-T_1)$ (where, if $T_2=+\infty$,
we set $T_2-T_1=+\infty$) and $\tilde{f}_0=e^{-T_1}f_{T_1}$.
\end{remark}

\begin{example}
Let $0<T_1<T_2< +\infty$ and let $0<\epsilon<T_2-T_1$. We construct a normal Loewner
chain on $\B^2$ which is ger\"{a}umig in $[T_1,T_2-\epsilon)$
but it is not ger\"{a}umig in $\R^+\setminus [T_1,T_2)$.
Let $\theta:\R^+\to [0,1]$ be such that $\theta(t)=1$ for $t\in [T_1,T_2]$
and $\theta(t)=0$ in $\R^+\setminus [T_1,T_2]$.
Define $G(z,t)=(-\theta(t)z_1-(1-\theta(t))(z_1-z_1^2), -z_2)$.
Then $G(z,t)$ is a Herglotz vector field associated with the class $\mathcal M$.
From the Loewner ODE $\frac{\de \v_{s,t}}{\de t}(z)=G(\v_{s,t}(z),t)$,
we find that for $s\in [T_1,T_2]$ it holds $\v_{s,T_2}(z)=e^{s-T_2}z$, while, for $t>T_2$
\[
\v_{T_2, t}(z)=e^{T_2-t}\left(
\frac{z_1}{1+(e^{T_2-t}-1)z_1}, z_2
\right).
\]
Therefore for $s\in [T_1,T_2]$ and $t>T_2$
\[
\v_{s,t}(z)=\v_{T_2,t}\circ  \v_{s,T_2}(z)=e^{s-t}
\left(
\frac{z_1}{1+(e^{s-t}-e^{s-T_2})z_1}, z_2
\right).
\]
By \cite[Theorem 8.1.5]{GK03}, the canonical solution to the Loewner PDE
associated with $G(z,t)$ is given by $(f_s)_{s\geq 0}$
with $f_s =\lim_{t\to\infty} e^t \v_{s,t}$. Hence, for $s\in [T_1,T_2]$,
\[
f_s(z)=\lim_{t\to \infty} e^t \v_{s,t}(z)=e^s
\left(
\frac{z_1}{1-e^{s-T_2}z_1}, z_2
\right).
\]
From this it follows easily that there exist $a,c>0$ such that $a\leq \mu(d(f_s)_z)$
and $\|(df_s)_z\|\leq c$ for all $z\in \B^2$ and for all $s\in [T_1,T_2-\epsilon]$.
Since $\|G(z,t)\|\leq 2$ for all $t\in \R^+$ and $z\in \B^2$,  by the Loewner PDE and
Remark \ref{equiv-cond}, it follows that $(f_s)_{s\geq 0}$ is ger\"{a}umig in $[T_1,T_2-\epsilon)$.
However, since $\lim_{z\to (1,0)}\Re \langle G(z,t), z\rangle=0$ for all $t\in \R^+\setminus [T_1,T_2]$,
the normal Loewner chain $(f_s)_{s\geq 0}$ is not exponentially squeezing in $\R^+\setminus [T_1,T_2)$,
hence it is not ger\"{a}umig in $\R^+\setminus [T_1,T_2)$.
\end{example}

\begin{theorem}
\label{squeez-geraumig}
Assume that $(f_t)_{t\geq 0}$ is a normalized Loewner chain, {\it respectively} a normal Loewner chain,
on $\B^n$. If $(f_t)_{t\geq 0}$ is exponentially squeezing  in $[T_1,T_2)$ for some $0\leq T_1<T_2< \infty$,  then  there exists a normalized Loewner chain, {\it respectively} a normal Loewner chain, $(g_t)_{t\geq 0}$
on $\B^n$ with $g_t=f_t$ for $t\in [0,+\infty)\setminus(T_1, T_2)$ (in particular, $g_0=f_0$)
and such that it is ger\"{a}umig  in
$[T'_1,T'_2)$ for every $T_1<T'_1<T_2'<T_2$.
\end{theorem}

\begin{proof}
Let $a\in (0,1]$ be the squeezing ratio of $(f_t)_{t\geq 0}$ in $[T_1,T_2)$.
Let $A\in (0,a)$.
Let $\alpha:\R^+\to [-(T_2-T_1)/2, 0]$ be an absolutely continuous function such that
$\alpha(t)=0$ for $t\in [0, T_1]\cup [T_2,+\infty)$,
$\al(t)=-A(t-T_1)$ for $t\in (T_1, T_1+(T_2-T_1)/2)$ and $\al(t)=A(t-T_2)$
for $t\in [T_1+(T_2-T_1)/2, T_2)$. Let
\[
g_t(z):=f_{t-\al(t)}( e^{\al(t)}z),
\quad  z\in \B^n,
\quad  t\geq 0.
\]
Then $(g_t)_{t\geq 0}$ is a family of holomorphic mappings on $\B^n$ such that
$g_t(0)=0$, $d(g_t)_0=e^t{\sf id}$, $t\mapsto g_t$ is a continuous mapping
with respect to the topology  in ${\sf Hol}(\B^n,\C^n)$
induced by the uniform convergence on compacta in $\B^n$, and $g_0=f_0$.
Moreover, notice that $g_t=f_t$ for $t\in [0,+\infty)\setminus(T_1, T_2)$.
Let $T_1<T'_1<T_2'<T_2$  be fixed.
By a direct computation, we have for all $z\in \B^n$ and a.e. $t\geq 0$
\begin{equation}\label{Herglotz-g}
[d(g_t)_z]^{-1}
\frac{\partial g_t}{\partial t}(z)=
e^{-\al(t)}(1-\al'(t))\left[df_{t-\al(t)}\right]^{-1}\frac{\de f_{t-\al(t)}}{\de t}(e^{\al(t)}z)+\al'(t)z
\end{equation}
Notice that $A<1$ implies  $t-\al(t)\in [T_1,T_2)$ for $t\in [T_1,T_2)$.
Therefore, setting $a(t)=0$ for $t\in [0,+\infty)\setminus[T_1, T_2)$
and $a(t)=a$ for $t\in [T_1,T_2)$ and taking into account that $(f_t)_{t\geq 0}$
is exponentially squeezing  in $[T_1,T_2)$ (with squeezing ratio  $a\in (0,1]$),
we obtain from (\ref{Herglotz-g}) that
\[
      \Re \left\langle [d(g_t)_z]^{-1}
\frac{\partial g_t}{\partial t}(z), \frac{z}{\|z\|^2}\right\rangle
\geq
(1-\al'(t))a(t)+\al'(t)\geq 0,
\quad \forall z\in \B^n,
\quad \hbox{a.e. }  t\geq 0.
\]
By Theorem \ref{all}.(3) and the fact that $g_t$ is univalent for all $t\geq 0$,
$(g_t)_{t\geq 0}$ is a normalized Loewner chain.
Since $A<a$, it follows easily that $(g_t)_{t\geq 0}$ is exponentially squeezing in  $[T_1',T_2')$.

Since $\al(t)\leq c<0$ for all $t\in [T_1',T_2']$, it follows easily that $(g_t)_{t\geq 0}$
satisfies (1) of
Definition \ref{spacious} for all $t\in [T_1',T_2']$.
From the definition of $g_t$, there exists a constant $c_1>0$ such that $\| d(g_t)_z\|\leq c_1$
for $z\in \B^n$,
a.e. $t\in [T'_1,T'_2]$.
Since $\mathcal{M}$ is compact,
using (\ref{Herglotz-g}) and the fact that
$[d(f_t)_z]^{-1}
\frac{\partial f_t}{\partial t}(z)\in \mathcal{M}$,
we conclude that there exists a constant $c_2>0$
such that $\left\|[d(g_t)_z]^{-1}
\frac{\partial g_t}{\partial t}(z)\right\|
\leq c_2$ for $z\in \B^n$,
a.e. $t\in [T'_1,T'_2]$.
Then, by Remark \ref{equiv-cond},
$(g_t)_{t\geq 0}$ also satisfies (2) of
Definition \ref{spacious}
and it is therefore ger\"aumig
in $[T'_1,T'_2)$.

Finally, if $(f_t)_{t\geq 0}$ is a normal Loewner chain, then $(g_t)$ is a normal
Loewner chain as well, because $g_t=f_t$ for $t\geq T_2$.
\end{proof}

\begin{remark}
If $(f_t)$ is a normalized/normal Loewner chain which is exponentially squeezing
in $[T_1,\infty)$ for some $T_1\geq 0$,
then for every $m>T_1$ the previous result allows to construct a  normalized/normal
Loewner chain $(g^m_t)$ which coincides with $(f_t)$ on $\R^+\setminus (T_1,m)$
and it is ger\"{a}umig in $(T_1',T_2')$ for all $T_1<T_1'<T'_2<m$.
\end{remark}

\section{Variation of ger\"{a}umig  Loewner chains}\label{vari}

\begin{theorem}\label{variation}
Assume that $(f_t)_{t\geq 0}$ is a normalized Loewner chain, {\it respectively} a normal Loewner chain,
on $\B^n$. If $(f_t)_{t\geq 0}$ is  ger\"{a}umig  in $[0,T)$ for some $T>0$,  then there exists  $\epsilon_0>0$ such that for all $\epsilon\in (0,\epsilon_0]$, setting
\[
\al(t):=\begin{cases}
\epsilon\left(1-\frac{t}{T}\right), \quad t\in [0,T)\\
0, \quad t\in [T,+\infty)
\end{cases}
\]
the family $(f_t(z)+\al(t)h(z))_{t\geq 0}$ is a normalized Loewner chain, {\it respectively}
a normal Loewner chain, on $\B^n$ for every $h: \B^n\to\C^n$ holomorphic with $h(0)=dh_0=0$ and $\sup_{z\in \B^n}\|h(z)\|\leq 1$, $\sup_{z\in \B^n}\|dh_z\|\leq 1$.
\end{theorem}

\begin{proof}
Let $c\in (0,1]$ be the squeezing ratio of $(f_t)_{t\geq 0}$ and let $a,b>0$ be given by Definition \ref{spacious}.
Up to replacing $a$ and $c$ with $\min\{a,c\}$, we can suppose $a=c$. Set $\epsilon_0=\min\left\{\frac{a}{2},\frac{a^3T}{2(a+bT)}\right\}$.

First of all, notice that $(f_t+\al(t)h)_{t\geq 0}$ is
a family of holomorphic mappings on $\B^n$ such that
$(f_t+\al h)(0)=0$, $d(f_t+\al h)_0=e^t{\sf id}$, and the mapping
$t\mapsto f_t+\al h$ is continuous
with respect to the topology in ${\sf Hol}(\B^n,\C^n)$
induced by the uniform convergence on compacta in $\B^n$.

Let $E\subset [0+\infty)$ be a set of full measure for which all
conditions in the hypotheses hold and such that $\frac{\de f_t}{\de t}$ exists for all $t\in E$.
First, we notice that $d(f_t)_z+\alpha(t) dh_z$ is invertible for all
$t\in E$ and all $z\in\B^n$. Indeed, it is so if  $t\geq T$. If $t\in E\cap [0,T)$
then fix $z\in\B^n$ and let $v\in \C^n$, $\|v\|=1$,
be such that $\mu(d(f_t)_z+\alpha(t) dh_z)=\|[d(f_t)_z+\alpha(t) dh_z](v)\|$.
It follows by (1) of Definition \ref{spacious}
\begin{equation*}
\begin{split}
\mu(d(f_t)_z+\alpha(t) dh_z)&=\|[d(f_t)_z+\alpha(t) dh_z](v)\|\geq \|d(f_t)_z(v)\|-\alpha(t)\|dh_z(v)\|\\&\geq \mu(d(f_t)_z)-\alpha(t)\|dh_z\|\geq a-\alpha(t)>0.
\end{split}
\end{equation*}
Hence, we can well define a vector field $G(z,t)$, holomorphic in $z\in \B^n$ and measurable in $t\in [0,+\infty)$ in the following way
\[
G(z,t):=\begin{cases} -[d(f_t)_z+\alpha(t)dh_z]^{-1}\left(\frac{\de f_t}{\de t}(z)+\alpha'(t)h(z)\right), & t\in E\\
0 & t\in [0,+\infty)\setminus E.
\end{cases}
\]
If $t\in E$, then $G(z,t)=-z+\sum_{k\geq 2} Q_k(z,t)$ where $Q_k$ is a polynomial mapping
in $z$ of order $k$. Hence $G(0,t)\equiv 0$ and $dG_0=-{\sf id}$.
We want to show that $-G(\cdot,t)\in \mathcal M$ for all
$t\in E$. For $t\geq T$ it is true because $(f_t)_{t\geq 0}$ is a
normalized Loewner chain, so we have to check the condition for $t\in E\cap [0,T)$.
To this aim, we first note that by Lemma \ref{linear},
$\|\alpha(t)(d(f_t)_z)^{-1}dh_z\|\leq \al(t)/a \leq \alpha(0)/ a\leq 1/2$
for all $t\in E\cap [0,T)$. Therefore, for  $t\in E\cap [0,T)$,
\begin{equation*}
\begin{split}
[d(f_t)_z+\alpha(t)dh_z]^{-1}&=[{\sf id}+\alpha(t)(d(f_t)_z)^{-1}dh_z]^{-1}[d(f_t)_z]^{-1}\\&=
\sum_{j=0}^\infty (-1)^j\alpha(t)^j[(d(f_t)_z)^{-1}dh_z]^j[d(f_t)_z]^{-1},
\end{split}
\end{equation*}
and
\begin{equation}\label{est1}
\begin{split}
\|[d(f_t)_z+\alpha(t)dh_z]^{-1}\|&\leq
\sum_{j=0}^\infty \alpha(t)^j\|(d(f_t)_z)^{-1}\|^{j+1}\|dh_z\|^j\\&\leq
\sum_{j=0}^\infty \frac{\alpha(t)^j}{ a^{j+1}}\leq \frac{2}{a}.
\end{split}
\end{equation}
While,
\begin{equation}\label{est2}
\begin{split}
\|[d(f_t)_z+\alpha(t)dh_z]^{-1}-[d(f_t)_z]^{-1}\|&\leq \sum_{j=1}^\infty \frac{\alpha(t)^j}{a^{j+1}}\\&=
\frac{\alpha(t)}{a^2}\frac{1}{1-\frac{\al(t)}{a}}\leq 2\frac{\alpha(t)}{a^2}\leq 2\frac{\epsilon_0}{a^2}.
\end{split}
\end{equation}
Now, since $\frac{\de f_t(0)}{\de t}=0$ for all $t\in E$, by the Schwarz
lemma and (2) of Definition \ref{spacious} it holds $\left\|\frac{\de f_t(z)}{\de t}\right\|\leq b\|z\|$.
Also, by the Schwarz lemma, $\|h(z)\|\leq \|z\|$. Hence, for all $t\in E\cap [0,T)$
\begin{equation*}
\begin{split}
\Re & \left\langle [d(f_t)_z+\alpha(t)dh_z]^{-1}\left(\frac{\de f_t}{\de t}(z)+
\alpha'(t)h(z)\right), \frac{z}{\|z\|^2}\right\rangle\\&=
\Re \left\langle [d(f_t)_z]^{-1}\frac{\de f_t}{\de t}(z), \frac{z}{\|z\|^2}\right\rangle
\\&+
\Re \left\langle ([d(f_t)_z+\alpha(t)dh_z]^{-1}-[d(f_t)_z]^{-1})\frac{\de f_t}{\de t}(z),\frac{z}{\|z\|^2}\right\rangle
\\&+
\Re\left\langle [d(f_t)_z+\alpha(t)dh_z]^{-1}\alpha'(t)h(z), \frac{z}{\|z\|^2}\right\rangle\\ &\geq a-\|[d(f_t)_z+\alpha(t)dh_z]^{-1}-[d(f_t)_z]^{-1}\|\left\|\frac{\de f_t}{\de t}(z)\right\|\frac{1}{\|z\|}\\&-|\al'(t)|\|[d(f_t)_z+\alpha(t)dh_z]^{-1}\|\|h(z)\|\frac{1}{\|z\|}\geq a-\frac{2b\epsilon_0}{a^2}-\frac{2\epsilon_0}{aT}\geq 0,
\end{split}
\end{equation*}
which proves that $\Re \left\langle G(z,t), \frac{z}{\|z\|^2}\right\rangle \leq 0$
for all $t\in E$ and $z\in \B^n\setminus\{0\}$ and by Remark \ref{simply},
$-G(\cdot,t)\in \mathcal M$ for all $t\in E$.

Hence, $G(z,t)$ is a Herglotz vector field associated with the class
$\mathcal M$ in $\B^n$ and $(f_t+\al(t)h)_{t\geq 0}$ is a normalized
solution to the Loewner-Kufarev PDE associated with $G(z,t)$.
In particular, it is a subordination chain by Theorem \ref{all}.
In order to prove that it is a normalized Loewner chain, we only need
to show that $f_t+\alpha(t)h$ is univalent for all $t\geq 0$.

Let $(g_t)_{t\geq 0}$ be the canonical solution associated with $G(z,t)$.
By Theorem \ref{all}, there exists a holomorphic mapping
$\Phi: \C^n\to \cup_{t\geq 0} (f_t+\alpha(t)h)(\B^n)$
such that $f_t(z)+\al(t)h(z)=\Phi(g_t(z))$ for all $t\geq 0$ and $z\in \B^n$.
Note that, for $t\geq T$, $\al(t)\equiv 0$, hence $f_t(z)=\Phi(g_t(z))$ for all $t\geq T$.
Taking into account that $\cup_{t\geq T}g_t(\B^n)=\C^n$ and $f_t$ is univalent for all $t\geq 0$,
it is easy to see that $\Phi$ is univalent. Therefore, $(f_t+\al(t)h)_{t\geq 0}$
is a normalized Loewner chain.

Finally, note that $\{e^{-t} (f_t+\alpha(t)h)\}_{t\geq 0}$ is a normal family if and only if
$\{e^{-t}f_t\}_{t\geq 0}$ is.
\end{proof}

Not all ``nice'' mappings in the class $S^0$ can be embedded into a normal Loewner chain
which is ger\"aumig in $[0,T)$ for some $T>0$,
as the following example shows:

\begin{example}\label{ese-no}
Let $a\in \C$ and let $f_a(z_1,z_2):=(z_1+az_2^2, z_2)$.
Note that $f_a$ is an automorphism of $\C^2$.
Let $g_a:=f_a|_{\B^2}$. It is known that $g_a\in S^0$ for $|a|\leq 3\sqrt{3}/2$
(see \cite[Example 3]{Su77}), while $g_a\not\in S^0$ for $|a|>2\sqrt{15}$
(see \cite[Remark 3.5]{GHKK}). Let $r_0:=\sup\{r\geq 0: g_r \in S^0\}$.
Then $3\sqrt{3}/2\leq r_0\leq 2\sqrt{15}$. Since $S^0$ is compact, $g_{r_0}\in S^0$.
If $g_{r_0}$ were embeddable in a normal Loewner chain
which is ger\"aumig in $[0,T)$ for some $T>0$,
then by Theorem \ref{variation} there would exist $\epsilon>0$
such that $g_{r_0+\epsilon}\in S^0$, contradicting the definition of $r_0$.
Note however that $g_{r_0}$ is embeddable
into a normalized Loewner chain (which is not a normal Loewner chain)
which is ger\"aumig in $[T_1,T_2)$ for any
$0\leq T_1<T_2<\infty$
given by $g_t(z):=f_{r_0}(e^t z)$.

In a recent paper \cite{B}, the first named author proved that,
in fact, $r_0=\frac{3\sqrt{3}}{2}$ and $g_{r_0}$ is a support point
of the class $S^0$, so that, according to Proposition \ref{p.support-extreme},
$g_{r_0}$ cannot be embedded into any normal Loewner chain which is exponentially squeezing
in some interval of $\R^+$.
\end{example}

\section{support points and extreme points}\label{support}

\begin{definition}
(i)
A mapping $f\in S^0$ is called a {\sl support point} if there exists a
linear operator $L:{\sf Hol}(\B^n, \C^n)\to \C$ which is continuous with respect to the
topology of uniform convergence on compacta of ${\sf Hol}(\B^n, \C^n)$ and not constant on $S^0$
such that $\max_{g\in S^0} \Re L(g)=\Re L(f)$. We denote by ${\sf Supp}(S^0)$ the set of support
points of $S^0$.

(ii)
A mapping $f\in S^0$ is called an {\sl extreme point}
if $f=tg+(1-t)h$, where $t\in (0,1)$, $g,h\in S^0$, implies
$f=g=h$.
We denote by ${\sf Ex}(S^0)$ the set of extreme points of $S^0$.

\end{definition}

\begin{lemma}\label{poly}
Let $L$ be a bounded linear operator on ${\sf Hol}(\B^n,\C^n)$ which is not constant on $S^0$.
Then there exists a polynomial mapping $h:\B^n\to \C^n$, $h(0)=0, dh_0=0$,
such that $\sup_{z\in \B^n} \|h(z)\|\leq 1$, $\sup_{z\in \B^n} \|dh_z\|\leq 1$ and $\Re L(h)>0$.
\end{lemma}
\begin{proof}
Since $L$ is not constant, there exists $f\in S^0$ such that $L(z)\neq L(f(z))$.
If $f=z+\sum_{j\geq 2}P_j(z)$ is the power series expansion of $f$ at $0$,
then $\sum_{j\geq 2}L(P_j(z))\neq 0$, hence there exists $P_j$ such that $L(P_j)\neq 0$.
Up to multiplication by a suitable complex number $\lambda$, we obtain the result.
\end{proof}

Although our variational method in Theorem \ref{variation} works only for normalized Loewner
chains which are  ger\"aumig in an interval $[0,T)$, $T>0$, it allows to prove the following result:

\begin{lemma}\label{one}
Let  $(f_t)_{t\geq 0}$ be a normal Loewner chain which is ger\"aumig in $[T_1,T_2)$ for some
$0\leq T_1<T_2\leq +\infty$. Then $f_0\not\in {\sf Supp}(S^0)\cup {\sf Ex}(S^0)$.
\end{lemma}

\begin{proof}
By \cite[Theorem 2.1]{GHKK1} and \cite[Theorem 1.1]{Schl}, if $f_0$ is an extreme point or a
support point, then so is $e^{-T_1}f_{T_1}$. Thus, it is enough to prove that
$e^{-T_1}f_{T_1}\not\in {\sf Supp}(S^0)\cup {\sf Ex}(S^0)$. By Remark \ref{back},
$e^{-T_1}f_{T_1}$ is embeddable into a normal Loewner chain which is ger\"aumig in $[0,T_2-T_1)$,
therefore, we can assume with no loss of generality that $T_1=0$.

Let $L$ be a bounded linear operator on ${\sf Hol}(\B^n,\C^n)$ which is not constant on $S^0$
and let $h$ be given by Lemma \ref{poly}. By Theorem \ref{variation},
there exists $\epsilon>0$ such that $f_0\pm\epsilon h\in S^0$.
But $\Re L(f_0+\epsilon h)=\Re L(f_0)+\epsilon \Re L(h)>\Re L(f_0)$, and $f_0$
is not a support point for $L$.

Also, since
\[
f_0=\frac{1}{2}(f_0+\epsilon h)+\frac{1}{2}(f_0-\epsilon h),
\]
$f_0$ is not an extreme point of $S^0$.
\end{proof}

Proposition \ref{p.support-extreme} follows then at once from
Theorem \ref{squeez-geraumig} and Lemma \ref{one}.

As we mentioned in the Introduction, our variational method can be
applied also to non-linear problem. As an example we prove the following result:

\begin{proposition}\label{nonlinear}
Let  $(f_t)_{t\geq 0}$ be a  normal Loewner chain which is exponentially
squeezing in $[T_1,T_2)$ for some $0\leq T_1<T_2\leq +\infty$.
Let $\Phi:{\sf Hol}(\B^n, \C^n)\to \C$ be a Fr\'echet differentiable
functional such that the Fr\'echet differential $L(f_0;\cdot)$ of $\Phi$
at $f_0\in S^0$ is  not constant on $S^0$. Then $f_0$ is not a maximum of $\Re \Phi$ in $S^0$.
\end{proposition}
\begin{proof}
Arguing as in the proof of \cite[Theorem 1.1]{Schl}, one can show that if $f=f_0$
is a maximum in $S^0$ for some Fr\'echet differentiable functional $\Phi$, then so is $e^{-T_1}f_{T_1}$
for the Fr\'echet differentiable functional $\Psi$,
where
\[
\Psi(g)=\Phi(e^{T_1}g\circ v_{T_1}),
\quad g\in {\sf Hol}(\B^n, \C^n),
\]
$v_t=v_{0,t}$, and $(v_{s,t})_{0\leq s\leq t}$ is the evolution family associated with $(f_t)_{t\geq 0}$.
Since
\begin{eqnarray*}
\Psi(e^{-T_1}f_{T_1}+g)&=&\Phi(f+e^{T_1}g\circ v_{T_1})
\\
&=&
\Psi(e^{-T_1}f_{T_1})+L(f;e^{T_1}g\circ v_{T_1})+o(\|e^{T_1}g\circ v_{T_1}\|),
\end{eqnarray*}
the Fr\'echet differential $L(e^{-T_1}f_{T_1};\cdot)$ of $\Psi$ at $e^{-T_1}f_{T_1}$ is
\[
L(e^{-T_1}f_{T_1};g)=L(f;e^{T_1}g\circ v_{T_1}),
\quad g\in {\sf Hol}(\B^n, \C^n).
\]
Since $L(f;\cdot):{\sf Hol}(\B^n, \C^n)\to \C$ is a  linear operator
which is continuous with respect to the topology of uniform convergence on
compacta of ${\sf Hol}(\B^n, \C^n)$ and is not constant on $S^0$,
arguing as in the proof of \cite[Proposition 2.6]{Schl},
one can show that $L(e^{-T_1}f_{T_1};\cdot)$ is not constant on $S^0$.
Thus, by Theorem \ref{squeez-geraumig} and Remark \ref{back},
we can assume with no loss of generality that $f$
is embeddable into a normal Loewner chain which is ger\"aumig in $[0,T)$ for some $T>0$.

Then, let $\epsilon_0>0$ be given by Theorem \ref{variation}.
Hence, for every $h: \B^n\to\C^n$ holomorphic with $h(0)=dh_0=0$ and
$\sup_{z\in \B^n}\|h(z)\|\leq 1$, $\sup_{z\in \B^n}\|dh_z\|\leq 1$ it follows that
$f+\epsilon h\in S^0$ for all $\epsilon\in [0,\epsilon_0)$. Let $h$ be given by Lemma \ref{poly}. Then
\[
\frac{1}{\epsilon}\Re [\Phi(f+\epsilon h)- \Phi(f)]= \Re L (f;h)+\Re \frac{o(\epsilon \|h\|)}{\epsilon}.
\]
Therefore, for $\epsilon$ sufficiently small, it follows that $\Re \Phi(f+\epsilon h)>\Re \Phi(f)$.
\end{proof}

We give some applications of Proposition \ref{p.support-extreme}.

\begin{proposition}\label{middle}
Let $g\in S^0$ and let $r\in (0,1)$. Also, let $f(z):=\frac{1}{r}g(rz)$. Then $f\in S^0$ and
$f\not\in {\sf Supp}(S^0)\cup  {\sf Ex}(S^0)$.
\end{proposition}

\begin{proof}
Let $(g_t)_{t\geq 0}$ be a normal Loewner chain such that $g_0=g$.
Set $f_t(z):=\frac{1}{r}g_t(rz)$. Then $(f_t)_{t\geq 0}$ is a normal Loewner chain such that $f_0=f$,
which thus belongs to $S^0$.

Next, recall that $\{g_t\}_{t\geq 0}$ satisfies the Loewner PDE
\begin{equation}\label{PDE}
\frac{\de g_t}{\de t}(z)=-d(g_t)_zG(z,t),\quad \hbox{ a.e. }\, t\geq 0,\, \forall z\in \B^n,
\end{equation}
where  $G(z,t)$ is a Herglotz vector field in $\B^n$ associated with the class $\mathcal M$.

Finally, since $-G(z,t)$ is in the class $\mathcal M$, for a.e. $t\geq 0$, it follows
from \eqref{PDE} that there exists $c>0$ (depending only on $r$) such that
for all $\|z\|\leq r$ and a.e. $t\geq 0$
\begin{equation}\label{mec}
\Re \left\langle [d(g_t)_z]^{-1}\frac{\de g_t}{\de t}(z), \frac{z}{\|z\|^2}\right\rangle \geq c.
\end{equation}
Indeed, let $E\subset [0,+\infty)$ be a set of full measure such that $G(z,t)$ is a
Herglotz vector field associated with the class $\mathcal M$, namely,
$\Re\left\langle G(z,t), \frac{z}{\|z\|^2}\right\rangle <0$ for all $z\in \B^n$ and $t\in E$.
Suppose by contradiction \eqref{mec} does not hold for some sequence $\{t_k\}_{k\in \N}\subset E$.
Then by \eqref{PDE} there exists a sequence $\{z_k\}_{k\in \N}\subset \B(0,r)$,
which we may suppose convergent to some $z_0$, with $\|z_0\|\leq r$ such that
$\Re\left\langle G(z_k,t_k), \frac{z_k}{\|z_k\|^2}\right\rangle>-1/k$.
Since $G(0,t_k)=0$ and $dG(0,t_k)=-{\sf id}$, clearly $z_0\neq 0$.
Let $v_k:=z_k/\|z_k\|$ and consider the holomorphic functions $g_k:\D \to \C$ defined by
\[
g_k(\zeta):=\langle G(\zeta v_k, t_k), v_k\rangle.
\]
Then $g_k(\zeta)=-\zeta p_k(\zeta)$ where $\Re p_k(\zeta)>0$ for all $\zeta\in \D$ and $p_k(0)=1$,
for all $k\in \N$. Then $\{p_k\}_{k\in \N}$ is a sequence of functions in the Carath\'eodory class.
In particular, since such a class is compact, we can assume that $p_k\to p$ for some holomorphic
function $p:\D \to \C$ with $p(0)=1$ and $\Re p(\zeta)>0$ for all $\zeta\in \D$.
But $\Re p_k(\|z_k\|)\to \Re p(\|z_0\|)$, which forces $\Re p(\|z_0\|)=0$, a contradiction.
Hence \eqref{mec} holds.

Now, since $[d(f_t)_z]^{-1}\frac{\de f_t}{\de t}(z)=\frac{1}{r}[d(g_t)_{rz}]^{-1}\frac{\de g_t}{\de t}(rz)$,
it follows from \eqref{mec} that for all $z\in \B^n\setminus\{0\}$ and a.e. $t\geq 0$, it holds
\[
\Re \left\langle [d(f_t)_z]^{-1}\frac{\de f_t}{\de t}(z),\frac{z}{\|z\|^2}\right\rangle=
\Re \left\langle [d(g_t)_{rz}]^{-1}\frac{\de g_t}{\de t}(rz), \frac{rz}{\|rz\|^2}\right\rangle \geq c.
\]
Hence $(f_t)_{t\geq 0}$ is exponentially squeezing  in $[0,+\infty)$.
By Proposition \ref{p.support-extreme}, it follows that $f\not\in {\sf Supp}(S^0)\cup {\sf Ex}(S^0)$.
\end{proof}

A class of mappings in $S^0$ which are not extreme/support points of $S^0$ may be
obtained in the following way:

\begin{proposition}
\label{c.ex-sup}
Let $(f_t)_{t\geq 0}$ be a normal Loewner chain and let $G(z,t)$ be the corresponding
Herglotz vector field associated with the
class ${\mathcal M}$. Assume that
$$G(z,t)=-[{\sf id}-E(z,t)]^{-1}[{\sf id}+E(z,t)](z),\quad z\in\B^n,\quad t\geq 0,$$
where $E(z,t)$ is an $(n\times n)$-matrix which is holomorphic with respect to $z\in\B^n$,
$E(0,t)=0$, for $t\geq 0$, and $E(z,t)$ is
measurable with respect to $t\in [0,\infty)$, for $z\in\B^n$.
If $\|E(z,t)\|\leq c<1$ for
$z\in\B^n$ and $t\geq 0$, then $f_0\not\in {\sf Supp}(S^0)\cup {\sf Ex}(S^0)$.
\end{proposition}

\begin{proof}
Since $\|E(z,t)\|\leq c$ and $E(0,t)=0$, it follows that $\|E(z,t)\|\leq c\|z\|$, by
the Schwarz lemma. Let
$$h(z,t):=[d(f_t)_z]^{-1}\frac{\de f_t}{\de t}(z)=[{\sf id}-E(z,t)]^{-1}[{\sf id}+E(z,t)](z).$$
Then it easily seen that
$$\|h(z,t)-z\|^2\leq c^2\|z\|^2\|h(z,t)+z\|^2,\quad z\in\B^n,\quad t\geq 0.$$
Now, elementary computations yield that
$$\|z\|^2\frac{1-c\|z\|}{1+c\|z\|}\leq \Re\langle h(z,t),z\rangle
\leq\|z\|^2\frac{1+c\|z\|}{1-c\|z\|},
\quad z\in\B^n,\quad t\geq 0,$$
and thus $(f_t)_{t\geq 0}$ is exponentially squeezing in $[0,+\infty)$.
The result follows from Proposition \ref{p.support-extreme}. This
completes the proof.
\end{proof}

\begin{corollary}\label{close}
Let $f\in S$ be such that  $\|df_z-{\sf id}\|\leq c$ for some $c\in (0,1)$ and for all $z\in \B^n$.
Then $f\in S^0$ and $f\not\in {\sf Supp}(S^0)\cup {\sf Ex}(S^0)$.
\end{corollary}

\begin{proof}
A normal Loewner chain with initial element $f$ is given by
$(f_t)_{t\geq 0}$ with $f_t(z)=f(e^{-t}z)+(e^t-e^{-t})z$ (see \cite[Proof of Lemma 2.2]{GHK}).
A direct computation shows that
\[
h(z,t):=[d(f_t)_z]^{-1}\frac{\de f_t}{\de t}(z)=[{\sf id}- E(z,t)]^{-1}[{\sf id}+E(z,t)](z),
\]
where $E(z,t)=e^{-2t}[{\sf id}-df_{e^{-t}z}]$. In view of the hypothesis, we deduce that
$\|E(z,t)\|\leq c$, for all $z\in\B^n$ and $t\geq 0$, and thus the result follows from
Proposition \ref{c.ex-sup}.
\end{proof}

\begin{remark}
Corollary \ref{close}
can be proved directly without inspecting the natural normal Loewner chain.
Indeed, if $L$ is a bounded linear functional
not constant on $S^0$ and $h$ is given by Lemma \ref{poly}, it is easy to see that
there exists $\epsilon>0$ such that $\|df_z\pm\epsilon dh_z-{\sf id}\|<1$ for all $z\in \B^n$.
Hence by \cite[Lemma 2.2]{GHK}, $f\pm\epsilon h\in S^0$ and then
$f\not\in{\sf Supp}(S^0)\cup {\sf Ex}(S^0)$.
\end{remark}

\section{Mappings that can be evolved in finite time to
a ball and time-$\log M$ reachable mappings}\label{Sreach}

A natural class of mappings in $S^0$ where our construction applies is that of
mappings whose image can be evolved to a ball:

\begin{definition}
Let $f$ be a normalized univalent mapping in $\B^n$.
We say that {\sl $f$ can be evolved in finite time $N>0$ to a ball}
if there exists a family $(f_t: \B^n \to \C^n)$ for $t\in [0,N]$
such that $f_t$ is univalent for all $t\in [0,N]$, $f_s(\B^n)\subset f_t(\B^n)$
for all $0\leq s\leq t\leq N$, $f_t(0)=0, d(f_t)_0=e^t {\sf id}$,
$f=f_0$ and $f_N(\B^n)=e^N\cdot\B^n$.
We denote by $\mathcal E_N$ the set of normalized univalent mappings in $\B^n$
that can be evolved in finite time $N>0$ to a ball.
\end{definition}

\begin{remark}\label{ev-to-ger}
Let $f$ be a normalized univalent mapping in $\B^n$.
Then it is easily seen that $f$ can be evolved in finite time $N>0$ to a ball if and only
if there exists a normal Loewner chain $(f_t)_{t\geq 0}$ which is ger\"{a}umig for $t>N$
such that $f=f_0$ and $f_N=e^N{\sf id}$.
\end{remark}
\begin{proof}
Indeed, if $f$ can be evolved in finite time $N>0$ to a ball by means of the
family $(f_t)_{t\in [0,N]}$, then setting $f_t(z)=e^t z$ for $t\geq N$
and $z\in \B^n$, it is easy to see that the family $(f_t)_{t\geq 0}$ is a normal Loewner
chain which is ger\"{a}umig for $t>N$. The converse statement is obvious.
\end{proof}

Moreover, let $M>1$ and denote by
\[S^0(M):=\{f\in S^0: \sup_{z\in\B^n}\|f(z)\|\leq M\}.\]

Then $\mathcal E_{\log M}\subset S^0(M)$, and also
\begin{equation*}
\mathcal E_M \subset \mathcal E_{M'}, \quad \forall M<M'.
\end{equation*}

\begin{remark}
(i) If $n=1$, then it is clear that $S^0(M)=S(M)$, where
$$S(M)=\{f\in S: |f(z)|<M, z\in \D\}.$$

(ii) In the case $n=1$, the family $\mathcal E_{\log M}$ coincides
with the family $S(M)$ (see \cite[Exercise 2, Chapter 6]{Po};
see also \cite{Go} and \cite{Pr}).
\end{remark}

The geometric notion of mappings that can be evolved in finite time to a ball has a
counterpart in control theory.
To see this, we first recall the following notion
(see {\sl e.g.} \cite{Pr}, \cite{Pr1}, \cite{Ro}, \cite{Ro1}, \cite{GHKK1}):
\begin{definition}
Let $f$ be a normalized univalent mapping in $\B^n$. We say that $f$ is
{\sl  time-$\log M$-reachable}
for some $M>1$ if there exists a Herglotz vector field $G(z,t)$
associated with the class $\mathcal M$ in $\B^n$ such that
$f=M \v(\cdot,\log M)$ where  $(\v(z,t))_{t\geq 0}$
is the solution to the Loewner ODE \eqref{LLODE} such that $\v(z,0)=z$.
The set of  time-$\log M$-reachable mappings generated by ${\mathcal M}$
is denoted by $\tilde{\mathcal R}_{\log M}({\rm id}_{\B^n},{\mathcal M})$.
\end{definition}

By \cite[Theorem 3.7]{GHKK1}, $f\in \tilde{\mathcal R}_{\log M}({\rm id}_{\B^n},{\mathcal M})$
for some $M>1$ if and only if $f$ can be evolved in time $\log M$ to a ball, {\sl i.e.},
\begin{equation}
\label{equal1}
\tilde{\mathcal R}_{\log M}({\rm id}_{\B^n},{\mathcal M})=\mathcal E_{\log M}.
\end{equation}
Thus, by \cite[Corollary 3.8]{GHKK1}, for every $N>0$, the set $\mathcal E_N$ is compact.

\begin{remark}
According to \cite[Corollary 7]{GHK2},
\begin{equation}
\label{imer}
\tilde{\mathcal R}_{\log M}({\rm id}_{\B^n},{\mathcal M})\subset S^0(M)\setminus ({\sf Supp}(S^0)\cup {\sf Ex}(S^0)).
\end{equation}
\end{remark}
The geometrical counterpart of \eqref{imer} follows at once either from \eqref{equal1}
or directly from Remark \ref{ev-to-ger} and Proposition \ref{p.support-extreme}:

\begin{corollary}
\label{cor-bounded}
Let $f$ be a normalized univalent mapping in $\B^n$ which can be evolved in
finite time $\log M$ to a ball, for some $M>1$.
Then $f\in S^0(M)\setminus( {\sf Supp}(S^0)\cup {\sf Ex}(S^0))$.
\end{corollary}

\begin{remark}
\label{r.reachable-nonlinear}
Let $f\in \mathcal{E}_{\log M}$, where $M>1$, and
let $\Phi:{\sf Hol}(\B^n, \C^n)\to \C$ be a Fr\'echet differentiable
functional such that the Fr\'echet differential $L(f;\cdot)$ of $\Phi$
at $f$ is  not constant on $S^0$. Taking into account Proposition
\ref{nonlinear} and Remark \ref{ev-to-ger}, we deduce that
$f$ is not a maximum of $\Re\Phi$ in $S^0$.
\end{remark}

As we already remarked, in dimension one
$S^0(M)=\mathcal E_{\log M}=\tilde{\mathcal R}_{\log M}({\rm id}_{\D},{\mathcal M})$ for all $M>1$.
In higher dimensions this is no longer the case. In order to properly state our result,
we need a preliminary result (see \cite[Example 3.5 and Theorem 3.12]{GHKK1}).

\begin{lemma}
\label{l-starlike}
Assume that $F\in S^0$ is starlike. For each $N>1$, define
\begin{equation}
\label{starlike}
F^N(z)=NF^{-1}\left(\frac{F(z)}{N}\right), \quad z\in \B^n.
\end{equation}
Then $F^N\in \mathcal{E}_{\log N}$. Moreover, if $F$
maximizes on $S^0$ a continuous functional $\lambda:S^0\to \R$ then  $F^N$
maximizes on $\mathcal{E}_{\log N}$  the  functional
$\lambda^N:\mathcal{E}_{\log N}\to\mathbb{R}$, defined by
$$\lambda^N(g)=\lambda(NF(N^{-1}g(\cdot))),\quad
g\in\mathcal{E}_{\log N}.$$ In
addition, $\lambda^N(F^N)=\lambda(F)$.
\end{lemma}

Let
$\Phi(z_1,z_2):=\Big(z_1+\frac{3\sqrt{3}}{2}z_2^2,z_2\Big)$, $z=(z_1,z_2)\in\B^2$.
Let $M:=2\sup_{z\in \B^2}\|\Phi(z)\|<+\infty$.
As we already remarked in Example \ref{ese-no},
$\Phi\in S^0(M)$ is a support point for $S^0$
and, in fact, it maximizes the functional
$\Re L^1_{0,2}: S^0\to \R$ given by $L_{0,2}^1(f)=
\frac{1}{2}\frac{\partial^2 f_1}{\partial z_2^2}(0)$
for $f=(f_1,f_2):\B^2\to\C^2$  holomorphic mapping (see  \cite{B}).
Such a mapping is starlike (see \cite{Su77}, \cite{B}). Let $N>1$. Then
\begin{equation}
\label{starlike3}
\Phi^N(z)=\left(z_1+\frac{3\sqrt{3}}{2}\left(1-\frac{1}{N}\right)z_2^2,z_2\right),\quad
z=(z_1,z_2)\in\B^2.
\end{equation}

Now we are ready to state and prove our result:
\begin{theorem}\label{PhiN}
For any $N>1$ the mapping $\Phi^N$ is a support point in $\mathcal E_{\log N}$
which maximizes the linear functional
$\Re L^1_{0,2}$. In particular, for all $f=(f_1,f_2)\in \mathcal E_{\log N}$ it follows
\begin{equation}
\label{coefficient-starlike}
\left|\frac{\partial^2 f_1}{\partial z_2^2}(0)\right|\leq 3\sqrt{3}\left(1-\frac{1}{N}\right),
\end{equation}
and this estimate is sharp. Moreover, for any $2<R<N$,
\[
\Phi^N\in S^0(M)\cup \mathcal E_{\log N}\setminus (\mathcal E_{\log R}\cup {\sf Supp}(S^0)\cup {\sf Ex}(S^0)).
\]
\end{theorem}
\begin{proof}
In view of Lemma \ref{l-starlike}, $\Phi^N\in \mathcal{E}_{\log N}$
maximizes the functional $(\Re L_{0,2}^1)^N: \mathcal{E}_{\log N}\to\mathbb{R}$.
A direct computation shows that, given $f=(f_1,f_2)\in \mathcal E_{\log N}$,
\[
(\Re L_{0,2}^1)^N (f)=\Re\frac{1}{2}\frac{\partial^2 f_1}{\partial z_2^2}(0)+
\frac{3\sqrt{3}}{2N}=\Re L_{0,2}^1(f)+\frac{3\sqrt{3}}{2N}.
\]
Hence, $\Phi^N$ maximizes $\Re L^1_{0,2}$ on  $\mathcal{E}_{\log N}$ and
\eqref{coefficient-starlike} follows at once.

In order to prove the last statement, we notice that,
given $1<R<N$, setting $r=(1-1/N)$, it follows that
$\Phi^N(z)=\frac{1}{r} \Phi(rz)$ for all $z\in \B^2$.
Hence, if $r\in (1/2,1)$---which amounts to $R>2$---we have
\[
\sup_{z\in \B^2}\|\Phi^N(z)\|\leq 2 \sup_{z\in \B^2}\|\Phi(z)\|=M.
\]
Therefore, $\Phi^N\in S^0(M)\cup \mathcal E_{\log N}$. Also, by Proposition
\ref{middle}, $\Phi^N\not\in  ({\sf Supp}(S^0)\cup {\sf Ex}(S^0))$.
Finally, $\Phi^N\not\in \mathcal E_{\log R}$ for all $R<N$ because otherwise
it would contradict \eqref{coefficient-starlike}.
\end{proof}

Note that the mapping $\Phi$ is a
bounded support point in $S^0$ which, by Corollary \ref{cor-bounded}, cannot be evolved in
finite time to a ball, and hence it is not time-$\log M$-reachable.
Therefore, such a mapping gives also a counterexample to Conjecture 3.9 in \cite{GHKK1}.

Moreover, and more interesting, Theorem \ref{PhiN} shows that, in contrast to the
one-dimensional case, in general
\[
\tilde{\mathcal R}_{\log M}({\rm id}_{\B^n},{\mathcal M})\not =
S^0(M)\setminus ({\sf Supp}(S^0)\cup {\sf Ex}(S^0)).
\]

It is also interesting to note that for every $R>N$,
\[
\Phi^N\in \mathcal{E}_{\log R}\setminus ({\sf Supp}(\mathcal{E}_{\log R})
\cup{\sf Ex}(\mathcal{E}_{\log R})).
\]
Indeed, in view of Theorem \ref{PhiN},
$\Phi^N\in \mathcal{E}_{\log N}\subset \mathcal{E}_{\log R}$.
Moreover, given any bounded linear functional $L$ on ${\sf Hol}(\B^n, \C^n)$,
it holds by linearity that $L(\Phi^N)= L({\sf id})+c_L\cdot L^1_{0,2}(\Phi^N)$
for some $c_L$. Hence, again by Theorem \ref{PhiN}, $\Phi^N\not\in {\sf Supp}(\mathcal{E}_{\log R})$.
Moreover, let $g(z_1,z_2)=(\delta z_2^2, 0)$ for some $\delta>0$. Note that $\Phi^N\pm g=\Phi^{N_\pm}$
for some $N_\pm >1$. If $\delta<<1$ then $\Phi^N+g, \Phi^N-g\in \mathcal{E}_{\log R}$.
Since $\Phi^N=\frac{1}{2}[(\Phi^N-g)+(\Phi^N+g)]$,
it follows that $\Phi^N\not\in {\sf Ex}(\mathcal{E}_{\log R})$

\begin{remark}
It is not known whether $\Phi$, (respectively $\Phi^N$), is an extreme point for $S^0$
(respectively for $\mathcal E_{\log N}$). However, since the
hyperplane $\{g\in {\sf Hol}(\B^2,\C^2): \Re L^1_{0,2}(g)=3\sqrt{3}/2\}$
intersects $S^0$ and it is a separating hyperplane (see {\sl e.g.} \cite[Theorem 4.6]{HM}),
there exists $f\in {\sf Ex}(S^0)$ such that
$\Re L^1_{0,2}(f)=3\sqrt{3}/2$. Similarly, for every $M>1$
there exists $f\in {\sf Ex}({\mathcal E}_{\log M})$ such that
$\Re L^1_{0,2}(f)=
\frac{3\sqrt{3}}{2}(1-1/M)$.
\end{remark}

The previous considerations make natural to ask the following questions:

\begin{question}
\label{q.embed}
Let $M>1$ and let $f\in S^0(M)\setminus ({\sf Supp}(S^0)\cup {\sf Ex}(S^0))$.
\begin{enumerate}
\item Is it true that there exists $R\geq M$ such that $f\in \mathcal E_{\log R}$?
\item  Is it true that $f$ can be embedded into an exponentially squeezing Loewner chain?
\end{enumerate}
\end{question}

Clearly, an affirmative answer to Question \ref{q.embed} (1) would imply an affirmative answer to
Question \ref{q.embed} (2).

\section*{Acknowledgments}
\label{acknowledgments}
The authors would like to thank the referees for valuable comments and suggestions
that improved the manuscript.

\bibliographystyle{amsplain}

\end{document}